\documentclass{amsart}

\usepackage{amsmath,amssymb}
\usepackage[utf8]{inputenc}
\usepackage[english]{babel}
\usepackage{enumerate}

\newtheorem{theorem}{Theorem}[section]
\newtheorem{lemma}[theorem]{Lemma}
\newtheorem{cor}[theorem]{Corollary}

\theoremstyle{definition}

\newtheorem{prop}[theorem]{Proposition}

\theoremstyle{remark}
\newtheorem{remark}[theorem]{Remark}

\numberwithin{equation}{section}

\newcommand{\Ace}{A_{c,\eps}}

\newcommand{\lce}{\lambda_{c,\eps}^{-1}}
\newcommand{\nce}{\eta_{c,\eps}}
\newcommand{\pce}{\pi^*_{c,\eps}}
\newcommand{\R}{\mathbb{R}}
\newcommand{\C}{\mathbb{C}}
\newcommand{\abs}[1]{\mathord{\left|#1\right|}}

\newcommand{\nrm}[1]{\mathord{\left\lVert #1 \right\rVert}}
\newcommand{\eps}{\varepsilon}
\newcommand{\lset}{\left\lbrace}
\newcommand{\rset}{\right\rbrace}
\newcommand{\F}[1]{\widehat{#1}}
\newcommand{\Digamma}{\frac{\Gamma'}{\Gamma}}
\newcommand{\half}{\frac{1}{2}}
\newcommand{\eand}{\,\,\,\text{ and }\,\,\,}
\newcommand{\zsum}{\sideset{}{^*}\sum}
\newcommand{\Fdx}{F_{\delta,x}}
\newcommand{\Gdx}{G_{\delta,x}}

\DeclareMathOperator{\Ren}{Re}
\DeclareMathOperator{\Imn}{Im}

\DeclareMathOperator{\li}{li}
\DeclareMathOperator{\Ei}{Ei}

\DeclareMathOperator{\E}{Ei}
\DeclareMathOperator{\supp}{supp}

\begin{document}

\title[Analytic Calculation of $\pi(x)$]{An improved analytic Method for calculating $\pi(x)$}

\author{Jan B\"uthe}
\address{Hausdorff Center for Mathematics, Endenicher Allee 62, 53115 Bonn, Germany}
\email{jan.buethe@hcm.uni-bonn.de}


\date{\today}

\dedicatory{}

\begin{abstract}
We provide an improved version of the analytic method of Franke et. al. for calculating the prime-counting function $\pi(x)$, which is more flexible and, for calculations not assuming the Riemann Hypothesis, also more efficient than the original method. The new method has recently been used to calculate the value $\pi(10^{25}) = 176,846,309,399,143,769,411,680$. 
\end{abstract}

\maketitle

\section{Introduction}

In \cite{LO87} Lagarias and Odlyzko presented an analytic algorithm which could calculate $\pi(x)$, the number of prime numbers not exceeding $x$, in run time $O(x^{1/2+\eps})$ for every $\eps>0$. The method is based on a modification of the well-known Perron formula, which reduces the calculation of $\pi(x)$ to numerically evaluating a complex curve integral involving the Riemann zeta function and calculating a correction term involving the powers of prime numbers in a neighbourhood of $x$. Although this method is asymptotically faster than the combinatorial Meissel-Lehmer-Lagarias-Miller-Odlyzko method \cite{LMO85}, which calculates $\pi(x)$ in run time $O(x^{2/3+\eps})$, the method was considered impractical at this time since the implied constant is large.

Recently, two new variants of the analytic method have been developed and implemented independently by Franke et. al. \cite{FKBJ} and Platt \cite{Platt15}, and for the first time records in the computation of $\pi(x)$ have been set with analytic methods: in 2010 Franke, Kleinjung, Jost and the author calculated the value 
\[
\pi(10^{24}) = 18,435,599,767,349,200,867,866,
\]
under the assumption of the Riemann Hypothesis (RH), which was later confirmed by an unconditional computation of Platt \cite{Platt15}.

The new methods are based on explicit formulas, replacing the curve integral in the original method by a sum over the non-trivial zeros of the Riemann zeta function. To achieve maximal efficiency the methods require the pre-computation of the zeros with imaginary part roughly up to $\sqrt{x}$ within an accuracy of $x^{-O(1)}$, which has conjectural run time $O(x^{1/2+\eps})$ \cite{OS88}. The evaluation of the sum over zeros is then much faster than evaluating the curve integral numerically, since the spacing between the zeros in consideration is comparably large.

The main difference between the methods in \cite{FKBJ,Platt15} is the choice of the kernel function, which is utilized to speed up the convergence in the Riemann explicit formula. Platt's method, an advancement of the work in \cite{Galway04}, uses the Gaussian function, while Franke chooses Logan's function \cite{logan88}. The latter satisfies an optimality condition which is very well-suited for this problem.

Although it could be shown in \cite{FKBJ} that the Logan function yields the more efficient algorithm, the restrictions on the parameters
in the original methods were too strong for unconditional calculations in certain situations, which led to the assumption of the RH in the calculation of $\pi(10^{24})$. This has been solved in an ad-hoc way by modifying the kernel function and introducing an additional parameter, giving rise to a second method in \cite{FKBJ}. In this paper we present a third and final version, which is less restrictive than the first method, simpler than the second method and generally more efficient for unconditional calculations.

 The new  method has also been implemented in cooperation with the authors of \cite{FKBJ}, and we calculated the value
\[
 \pi(10^{25}) = 176,846,309,399,143,769,411,680
\]
without assuming the RH. We used the zeros with imaginary part up to $10^{11}$ for calculation, which resulted in a run time of $40,000$ CPU hours. The run time could have been reduced considerably by the use of additional zeros.

\tableofcontents

\section{Description of the method}\label{s:pice}

Let $\pi^*(x) = \half \sum_{p^m < x} \frac{1}{m} + \half\sum_{p^m\leq x} \frac{1}{m}$ denote the normalized Riemann prime-counting function. This function is related to the non-trivial zeros of the Riemann zeta function by the well-known Riemann explicit formula 
\begin{equation}\label{e:Riemann-exp-form}
\pi^*(x) = \li(x) - \zsum_{\substack{\rho\\\Im(\rho)>0}}\Ei(\rho\log x) \\
 - \log(2) + \int_{x}^\infty \frac{dt}{t\log(t)(t^2-1)}
\end{equation}
\cite{Riemann59,Mangoldt95}.  Here $\Ei(z)$ is the antiderivative of $e^z/z$ in $\C\setminus(-\infty,0]$ satisfying
\[
\Ei(1) = \lim_{\eps\searrow 0} \Bigl(\int_{-\infty}^{-\eps} + \int_{\eps}^1\Bigr) \frac{e^\xi}{\xi}\, d\xi,
\]
$\li(x) = \Ei(\log(x))$ denotes the logarithmic integral, and the star on the sum over zeros indicates that it is calculated as
\[
\lim_{T\to \infty} \sum_{\abs{\Im(\rho)}<T}.
\]

Calculating $\pi(x)$ by approximating the sum over zeros in \eqref{e:Riemann-exp-form} is inefficient, since the sum converges too slowly \cite{RG70}. This problem can be solved by considering continuous approximations to $\pi^*(x)$, satisfying an explicit formula wherein the sum over zeros converges fast.

Such approximations can be constructed using the Weil-Barner explicit formula \cite{weil52,barner81,Lang86} for the Riemann zeta function. Suppose that $g\colon \R \rightarrow \C$ satisfies the following conditions:
\vspace{.2cm}
\begin{enumerate}[(B1)]
\item there exists a $h>0$ such that $g(t) e^{(\half + h)\abs{t}}$ is of bounded variation on $\R$, \label{en:BV}
\item there exists an $\eps>0$ such that $2 g(0) = g(t) + g(-t) + O(\abs{t}^\eps)$ for $t\to 0$,\label{en:origin}
\item $g(t) = \half\lim_{h\searrow 0} \bigl(g(t+h)  + g(t-h)\bigr)$ for all $t\in \R$.\label{en:normalized}
\end{enumerate}
\vspace{.2cm}
Then the Weil-Barner explicit formula is given by
\begin{equation}\label{e:WB-exp-form}
w_s(\hat g) = w_f(g) + w_\infty(g),
\end{equation}
where
\begin{align}
w_s (\hat g) &= \zsum_\rho \hat g \Bigl( \frac \rho i - \frac 1{2i}\Bigr) - \hat g (i/2) - \hat g(-i/2),\label{d:ws}\\
w_f (g) &= -\sum_p \sum_{m=1}^\infty \frac{\log p}{p^{m/2}}\bigl( g(m\log p) + g(-m\log p)\bigr)\label{d:wf}, \\
\intertext{and}
w_\infty(g) &= \Bigl(\frac{\Gamma'}{\Gamma}(1/4) - \log \pi\Bigr)g(0) -\int_{0}^\infty \frac{g(t) + g(-t)-2g(0)}{1-e^{-2t}}e^{-t/2}\, dt,\label{d:w8}
\end{align}
and where
\begin{equation*}
\hat g(\xi) = \int_{-\infty}^\infty g(t) e^{i\xi t}\, dt
\end{equation*}
denotes the Fourier transform of $g$. The class of test functions satisfying (B\ref{en:BV}) - (B\ref{en:normalized}) is usually referred to as the Barner class.

The Riemann explicit formula follows from \eqref{e:WB-exp-form} by considering a suitable approximating sequence for the function $\chi^*_{(-\infty,\log x]}(t) e^{t/2}/t$, where
\begin{equation*}
\chi^*_A (t) = 
\begin{cases}
0	& t\notin \overline{A} \\
1 & t\in A\setminus\partial A \\
1/2	&t\in \partial A
\end{cases}
\end{equation*}
denotes the normalized characteristic function \cite{BFJK13}. The slow convergence in \eqref{e:Riemann-exp-form} is caused by the discontinuity at $t=\log x$, which we intend to remove by taking the convolution with the Fourier transform of the Logan function
\[
\ell_c(t) = \frac{c}{\sinh c} \frac{\sin(\sqrt{t^2-c^2})}{\sqrt{t^2-c^2}}.
\]
This function minimizes the functional
\begin{equation}\label{e:logan-opt}
\int_{\abs{t}>c} \abs{\frac{f(t)}{t}}\, dt
\end{equation}
among a suitable class of test functions with $f(0)=1$ and $\supp \hat f \subset [-1,1]$, 
the minimal value being $2\log \frac{1+e^{-c}}{1-e^{-c}}$ \cite{logan88}. This will be beneficial for truncating the sum over zeros in the explicit formula where the size of the remainder will be controlled by the parameter $c$. Explicitly, the inverse
  Fourier transform of $\ell_c$ is given by
\[
 \eta_{c}(y) = \chi^*_{[-1,1]}(y) \frac{c}{2\sinh c} I_0(c\sqrt{1-y^2}),
\]
where $I_0 (y) = \sum_{n=0}^\infty (y/2)^{2n}/(2n)!$ denotes the $0$-th modified Bessel function (see \cite{FKBJ}).
Since we will be dealing with dilations of $\ell_c$ we introduce the shorter notations
\begin{equation*}
\eta_{c,\eps}(y) =\frac 1\eps \eta_c(y/\eps), \eand \ell_{c,\eps}(t) := \hat \eta_{c,\eps}(t) =\ell_{c}(\eps t).
\end{equation*}

Now let $f_k(t) = e^{t/2}/t^k$,  $\lambda_{c,\eps} = \ell_{c,\eps}(i/2)$, and $A_{c,\eps} = -\ell''_{c,\eps}(0)/2$,
and for $\abs{t}>\eps$ let
\begin{equation}\label{e:def-phi}
 \phi_{x,c,\eps}(t) = \lce\Bigl(\chi_{(-\infty,\log x]}\Bigl(f_1 +A_{c,\eps}\Bigl( f_2 - 2f_3\Bigr)\Bigr)\Bigr)\ast \nce(t),
\end{equation}
where the convolution operator $\ast$ is defined as usual by
\begin{equation*}
f\ast g(x) = \int_{-\infty}^\infty f(y) g(x-y)\, dy.
\end{equation*}
Then, for all $c>0$ and $0<\eps<\log 2$ we define the modified prime-counting function
\begin{equation}\label{e:mod-pi}
 \pi^*_{c,\eps}(x) =  \sum_{p^m} \frac{\log p}{p^{m/2}}\phi_{x,c,\eps} (m\log p).
\end{equation}
Under certain mild restrictions on the parameters, we will show that
\begin{itemize}
\item $\pi^*(x)$ can be calculated within an accuracy of $O(\eps^3 x)$ from $\pi^*_{c,\eps}(x)$ by evaluating a sum over the powers of prime numbers in $[e^{-\eps} x, e^{\eps} x]$
 \item $\pi^*_{c,\eps}(x)$ can be calculated within an accuracy of $O(e^{-c} x^h\log\log(x))$ using the zeros with imaginary part up to $c/\eps$, where $h=1/2$ if the RH is assumed and $h=1$ otherwise.
\end{itemize} 
So for $c \geq h\log(x) + \log\log\log(x)  +  C_1$ and $\eps < C_2x^{-1/3}$, with suitable constants $C_1$ and $C_2$ the exact value of $\pi(x)$ can be determined from the approximation to $\pi^*(x)$ by removing the contribution of higher prime powers. These results can be improved for calculations not assuming the RH, which increases the method's efficiency compared to the methods in \cite{FKBJ}.

Since the remaining part of this paper involves many explicit estimates, we will use Turing's big theta notation and 
write \textit{$f = \Theta(g)$ in $U$} if $\abs{f(t)} \leq g(t)$ holds for all $t\in U$.

\section{The difference $\pi^*(x)-\pi_{c,\eps}^*(x)$}\label{s:diff}
The main contribution to the difference $\pi^*(x)-\pce(x)$ is given by the prime powers in $[e^{-\eps}x,e^\eps x]$
 and may be expressed in terms of the functions
\[
\mu_{c,\eps}(t) =
\begin{cases}
-\int_{-\infty}^t\nce(\tau)\,d\tau	& t < 0\\
-\mu_{c,\eps}(-t)			&t>0\\
 0					&t=0
\end{cases}
\]
and
\[
\nu_{c,\eps}(t) = \int_{-\infty}^t \mu_{c,\eps}(\tau)\, d\tau.
\]


\begin{theorem}\label{t:pice-to-pi}
Let $x>e^{10}$, $c\geq 1,$ $x^{-1}<\eps \leq 10^{-3}$, and  let
\begin{equation*}
M_{x,c,\eps}(t) = \lce\Bigl[\mu_{c,\eps}\bigl(\log\tfrac{t}{x}\bigr) + \Bigl(\frac{1}{\log t} - \frac{1}{2}\Bigr)\Bigl(\mu_{c,\eps}\bigl(\log\tfrac{t}{x}\bigr)\log\tfrac{t}{x} - \nu_{c,\eps}\bigl(\log\tfrac{t}{x}\bigr)\Bigr)\Bigr].
\end{equation*}
Then we have
\begin{equation}\label{e:pice-to-pi}
\pi^*_{c,\eps}(x) = \pi^*(x) + \sum_{e^{-\eps}x <p^m<e^{\eps} x}\frac{1}{m}M_{x,c,\eps}(p^m) + R(x,c,\eps),
\end{equation}
where
\begin{equation}\label{e:R-bound0}
\abs{R(x,c,\eps)} \leq \frac{0.57\, \eps^3 x}{c\log(\eps x)} + \frac{39\, \eps^4 x}{c^3} + \frac{0.13 \, \eps^2 \log \log(2x^2)}{c}.  
\end{equation}
\end{theorem}

\begin{remark}
Since the method will require $c\geq \frac 12 \log x$, the error term in \eqref{e:R-bound0} will be negligible for $ \eps < \delta \log(x)^{2/3}/x^{1/3}$ with some $\delta>0$. This upper bound for $\eps$ translates into a lower bound of size $O((x\log x)^{1/3})$ for the truncation bound in the sum over zeros, i.e. the number $B$ such that contributions of zeros $\rho$ with $\abs{\Im(\rho)}>B$ will be discarded. Since the optimal truncation bound is rather of size $\sqrt{x}$, this does not impose any practical restrictions.
\end{remark}

The proof needs some preparation. First we estimate the difference $\phi_{\infty,c,\eps} - f_1(t)$, where $\phi_{\infty,c,\eps} = \lim_{x\to\infty} \phi_{x,c,\eps}$.


\begin{lemma}\label{l:phi-infty}
Let $\eps \leq 0.001$, $c\geq 1$ and $\abs{t}\geq\log 2$. Then we have
\begin{equation}\label{e:phi-infty}
\phi_{\infty,c,\eps}(t) = f_1(t) + \Theta\Bigl(39\frac{\eps^4}{c^2}f_2(t)\Bigr).
\end{equation} 
\end{lemma}


\begin{proof}
Let
\[
g_k(\tau) = e^{-\tau/2} \Bigl(\frac{t^2}{(t-\tau)^k}-t^{2-k}\Bigr).
\]
Then we have
\begin{equation}\label{e:fk*nce}
f_k\ast\nce(t) = \int_{-\eps}^\eps \nce(\tau) \frac{e^{\frac{t-\tau}{2}}}{(t-\tau)^k}\,d\tau
= \lambda_{c,\eps}f_k(t) + f_2(t) \int_{-\eps}^\eps \nce(\tau) g_k(\tau)\, d\tau.
\end{equation}
The integral on the right hand side can be expressed in terms of the derivatives $\ell_{c,\eps}^{(n)}(0)$ by expanding $g_k$ in a Taylor series and using the well-known identity
\begin{equation*}
\ell_{c,\eps}^{(n)} (0) = i^n\int_{\eps}^\eps \nce(\tau) \tau^n\, d\tau.
\end{equation*}
 Since $\nce$ is even, we only need to consider even terms. We find
\begin{align*}
g_k(0) &=0, \\
g''_k(\tau) &= e^{-\tau/2}\Bigl[-\frac{t^{2-k}}{4} + \frac{1}{4}\frac{t^2}{(t-\tau)^k} - \frac{kt^{k+1}}{(t-\tau)^2} + \frac{k(k+1)t^2}{(t-\tau)^{k+2}}\Bigr],\\
\intertext{and}
g_1^{(4)}(\tau) &= e^{-\tau/2}\Bigl[\frac 1{16}\frac{t\tau}{t-\tau} - \frac 12 \frac{t^2}{(t-\tau)^2} + \frac{3t^2}{(t-\tau)^3} - \frac{12t^2}{(t-\tau)^4} + \frac{24t^2}{(t-\tau)^5}\Bigr],
\end{align*}
and since we may assume $\abs{t}\geq \log 2$ and $\abs{\tau}\leq \eps \leq 0.001$ we have
\begin{equation*}
\abs{\frac{t}{t-\tau}} \leq \frac{\log 2}{\log 2 - 0.001} \leq 1.002,
\end{equation*}
which gives
\begin{equation*}
\abs{g''_k(\tau)} \leq 1.002^{k+2} e^{0.0005}\Bigl[\frac{1}{2\log(2)^{k-2}} + \frac{k}{\log(2)^{k-1}}+\frac{k(k+1)}{\log(2)^{k}}\Bigr]\leq
\begin{cases}
16.1 & k=2,\\
43.5 & k=3,
\end{cases}
\end{equation*}
 and
\begin{equation*}
\abs{g_1^{(4)}(\tau)} \leq 1.002^5\cdot e^{0.0005}  \Bigl[\frac{1}{16}0.001 + \frac{1}{2} + \frac{3}{\log(2)} + \frac{12}{\log(2)^2} + \frac{24}{\log(2)^3}\Bigr] \leq 103. 
\end{equation*}
We therefore get
\begin{align*}
g_1(\tau) &= \alpha_1\tau + \frac{\tau}{2}\Bigl(\frac{2}{t^2}-\frac{1}{t}\Bigr) + \beta_1 \tau^3 + \Theta(4.3\tau^4), \\
g_2(\tau) &= \alpha_2\tau + \Theta(8.1\tau^2),
\intertext{and}
g_3(\tau) &= \alpha_3 \tau + \Theta(21.8\tau^2),
\end{align*}
 which gives
\begin{equation}\label{e:g_k-int}
\abs{\int_{-\eps}^\eps \nce(\tau) g_k(\tau) \,d\tau} \leq
\begin{cases}
8.1 \,\abs{\ell''_{c,\eps}(0)} & k=2, \\
21.8 \,\abs{\ell''_{c,\eps}(0)} & k = 3,
\end{cases}
\end{equation}
and
\begin{equation}\label{e:g_1-int}
\int_{-\eps}^\eps \nce(\tau) g_1(\tau) \,d\tau = -\frac{\ell''_{c,\eps}(0)}{2}\Bigl(\frac{2}{t^2}-\frac{1}{t}\Bigr) 
+ \Theta(4.3 \, \ell_{c,\eps}^{(4)} (0)).
\end{equation}

Next we estimate the derivatives $\ell''_{c,\eps}(0)$ and $\ell^{(4)}_{c,\eps}(0)$. Since $c\geq 1$, we have
\begin{equation*}
0 \leq \frac{\cosh(c)}{\sinh(c)}-\frac{1}{c} = 1 + \frac{e^{-c}}{\sinh(c)} - \frac{1}{c} \leq 1 + \frac{1}{c}\Bigl(\frac{e^{-1}}{\sinh(1)} - 1\Bigr) \leq 1
\end{equation*}
and since $(-1)^n \ell_{c,\eps}^{(2n)}(0)>0$ we obtain the bounds
\begin{equation}\label{e:logan-der2}
0 < -\ell_{c,\eps}''(0) = \frac{\eps^2}{c}\Bigl(\frac{\cosh(c)}{\sinh(c)} -\frac{1}{c}\Bigr) \leq \frac{\eps^2}{c},
\end{equation}
and
\begin{equation*}
0 < \ell_{c,\eps}^{(4)}(0) 
= \frac{9\eps^4}{c^2}\Bigl(\frac{1}{c^2} - \frac{\cosh(c)}{c\sinh(c)} +\frac{1}{3}\Bigr) 
\leq \frac{3\eps^4}{c^2}.
\end{equation*}

Combining these with the estimates in \eqref{e:g_k-int} and \eqref{e:g_1-int} now gives
\begin{align*}
\phi_{\infty,c,\eps}(t) &= \lce\bigl(f_1 + A_{c,\eps} (f_2-2f_3)\bigr)\ast \nce(t) \\
&= f_1(t)\Bigl[1 + \Theta\Bigl(4.3\frac{3\eps^4}{tc^2}\Bigr) + A_{c,\eps} \Theta\Bigl((8.1+43.6)\frac{\eps^2}{tc}\Bigr)\Bigr] \\
&= f_1(t)\Bigl[1 + \Theta\Bigl(39\frac{\eps^4}{tc^2}\Bigr)\Bigr],
\end{align*}
where we used $\lambda_{c,\eps}>1$.
\end{proof}


Next, we bound the difference $\chi^*_{[\log 2,\log x]}\phi_{\infty,c,\eps} -  \phi_{x,c,\eps}$ in $B_\eps(\log x)$.


\begin{lemma}\label{l:phi-at-logx}
Let $x\geq e^{10}$, $\eps \leq 0.001$, $c\geq 1$, and let
\begin{equation}\label{e:m-def}
m_{x,c,\eps}(t) = \frac{e^{t/2}}{\lambda_{c,\eps} t}\Bigl[\mu_{c,\eps}(y) + \Bigl(\frac{1}{t}-\frac{1}{2}\Bigr)\bigl(y\mu_{c,\eps}(y) - \nu_{c,\eps}(y)\bigr)\Bigr],
\end{equation} 
where $y = t-\log(x)$. Then we have
\begin{equation*}
\phi_{x,c,\eps}(t) = \chi^*_{[1,\log x]}(t)\phi_{\infty,c,\eps}(t) + m_{x,c,\eps}(t) + \Theta\Bigl(0.13\, e^{-\eps/2}\frac{\eps^2\sqrt{x}}{c \log(e^\eps x)}\Bigr)
\end{equation*}
for $\abs{t-\log x} \leq \eps$.
\end{lemma}


\begin{proof}
For $t\in B_\eps(\log x)$ we have
\begin{multline}\label{e:fk-si}
(\chi_{[1,\log x]}f_k)\ast \nce(t) 
= \chi^*_{[\log x, \infty)}(t) \int_{y}^\eps \nce(\tau) f_k(t-\tau) \, d\tau \\
+ \chi^*_{[1,\log x]}(t) \Bigl[ f_k\ast\nce(t) - \int_{-\eps}^y \nce(\tau) f_k(t-\tau) \, d\tau\Bigr]
\end{multline}
Since
\[
0<f_k(t) \leq f_k(\log(x) + \eps) \leq e^{\eps/2}\frac{\sqrt{x}}{\log(x)^k}
\]
for such $t$, this gives
\begin{equation*}
(\chi_{[1,\log x]}f_k)\ast \nce(t) = \chi_{[1,\log x]}^*(t)f_k(t) + \Theta\Bigl(\frac{e^{\eps/2}\sqrt{x}}{2\log(x)^2}\Bigr),
\end{equation*}
where we used  $\int_0^\eps \nce(\tau) d\tau = 1/2$ in the $\Theta$-term.

For $k=1$ we further evaluate the integrals in \eqref{e:fk-si}. Since $t\geq 10$, we have
\[
\frac{e^{\frac{t-\tau}{2}}}{t-\tau} 
= f_1(t)\Bigl(1 + \tau\Bigl(\frac{1}{t} - \frac{1}{2}\Bigr) + \Theta(0.13\, \tau^2)\Bigr).
\]
Now first assume $y > 0$. Then we get
\begin{align*}
\int_{y}^{\eps} \nce(\tau) f_1(t-\tau)\,d\tau 
&= \frac{e^{\frac{t}{2}}}{t}\int_{y}^\eps \nce(\tau) (1 + \tau(\tfrac{1}{t} - \tfrac{1}{2}) + \Theta(0.13\tau^2))\,d\tau\\
&=  \frac{e^{\frac{t}{2}}}{t}\Bigl( \mu_{c,\eps}(y) + (\tfrac{1}{t} - \tfrac{1}{2})\int_{y}^\eps \nce(\tau) \tau\,d\tau + \Theta\Bigl( 0.065\frac{\eps^2}{c}\Bigr)\Bigr)\\
&= \frac{e^{\frac{t}{2}}}{t}\Bigl( \mu_{c,\eps}(y) + (\tfrac{1}{t} - \tfrac{1}{2})(y\mu_{c,\eps}(y) - \nu_{c,\eps}(y)) + \Theta\Bigl(0.065\frac{\eps^2}{c}\Bigr)\Bigr)
\end{align*}
for the first integral in \eqref{e:fk-si}, where we used the bound from \eqref{e:logan-der2} again. A similar computation gives
\begin{equation*}
\int_{-\eps}^y \nce(\tau) f_1(t-\tau)\,d\tau = -f_1(t) \Bigl(\mu_{c,\eps}(y) + (\tfrac{1}{t} - \tfrac{1}{2})(y\mu_{c,\eps}(y) - \nu_{c,\eps}(y)) + \Theta\Bigl(0.065\frac{\eps^2}{c}\Bigr)\Bigr)
\end{equation*}
for $y < 0$. This also includes the case $y=0$ since $\mu_{c,\eps}$ and $\nu_{c,\eps}$ are normalized. The assertion now follows from
\[
	e^{\eps/2} \frac{\sqrt{x}}{\log x} \Bigl[0.065 \frac{\eps^2}{c} + A_{c,\eps}\Bigl(\frac{1}{\log(x)} + \frac{2}{\log(x)^2}\Bigr)\Bigr] \leq 0.13\, e^{-\eps/2}\frac{\eps^2\sqrt{x}}{c\log(e^\eps x)}.\qedhere
\]
\end{proof}


\begin{proof}[Proof of Theorem \ref{t:pice-to-pi}]
Let $I=[e^{-\eps}x, e^\eps x]$. By Lemma \ref{l:phi-at-logx} we have
\begin{align}\label{e:proof-start}
 \pce(x) & = \sum_{p^m} \frac{\log p}{p^{m/2}}\phi_{x,c,\eps}(m\log x) \\
& = \sum_{p^m} \frac{\log p}{p^{m/2}} \bigl(\chi^*_{[1,\log x]} \phi_{\infty,c,\eps}\bigr)(m\log p) \notag \\
  &\quad\quad + \sum_{p^m\in I} \frac{\log p}{p^{m/2}}\Bigl( m_{x,c,\eps}(m\log p) + \Theta\Bigl(0.13 \frac{e^{-\eps/2}\eps^2\sqrt{x}}{c\log(e^{\eps} x)}\Bigr)\Bigr).\notag
\end{align}

The sum on the second line of \eqref{e:proof-start} may be written as
\begin{equation*}
 \pi^*(x) + \sum_{p^m} \frac{\log p}{p^{m/2}}\chi^*_{[1,\log x]}(\phi_{\infty,c,\eps} - f_1)(m\log p),
\end{equation*}
and the bound from Lemma \ref{l:phi-infty} gives
\begin{equation*}
\sum_{p^m\leq x}\frac{\log p}{p^{m/2}} \abs{(\phi_{\infty,c,\eps} - f_1)(m\log p)} \leq 39\frac{\eps^4}{c^2}\sum_{p^m\leq x} \frac{1}{m^2\log(p)} 
\leq 39\frac{\eps^4}{c^2} x.
\end{equation*}

Since
\[
 \frac{\log p}{p^{m}} \leq e^{\eps/2} \frac{\log(e^\eps x)}{\sqrt{x}},
\]
for all $p^{m}\in I$, the third line of \eqref{e:proof-start} takes the form
\begin{equation*}
 \sum_{p^m\in I} \frac{1}{m}M_{x,c,\eps}(p^{m}) + \Theta\Bigl(0.13 \frac{\eps^2}{c} \sum_{p^m\in I} \frac{1}{m}\Bigr).
\end{equation*}
It therefore suffices to show
\begin{equation*}
0.13 \frac{\eps^2}{c} \sum_{p^m\in I} \frac{1}{m} \leq \frac{0.57\eps^3 x}{c\log(\eps x)} + \frac{0.13 \eps^2 \log\log(2x^2)}{c}.
\end{equation*}
But the Brun-Titchmarsh inequality, as stated in \cite{MV73}, gives
\begin{equation}\label{e:t1-prime-cont}
0.13\frac{\eps^2}{c} \sum_{p\in I} 1 \leq \frac{0.13\eps^2}{c} \frac{4.01\eps x}{\log(\eps x)}\leq 0.53 \frac{\eps^3 x}{c \log(\eps x)},
\end{equation}
and since $\log(\eps x) \sqrt{x} \leq 10\, e^{-5} x$ for $x\geq e^{10}$, Lemma \ref{l:ppow-bound} below gives
\begin{equation*}
0.13\frac{\eps^2}{c} \sum_{\substack{p^m\in I}\\m\geq 2} \frac 1m \leq 0.04 \frac{\eps^3 x}{c \log(\eps x)} + \frac{0.13 \eps^2\log\log(2x^2)}{c},
\end{equation*}
which yields the assertion.
\end{proof}

\begin{lemma}\label{l:ppow-bound}
Let $x\geq 100$, $\eps\leq \frac 1{100}$ and let $I=[e^{-\eps}x,e^\eps x]$. Then we have
\begin{equation*}
\sum_{\substack{p^m\in I\\ m\geq 2}} \frac{1}{m} \leq 4.01 \eps \sqrt x + \log\log(2x^2).
\end{equation*}
\end{lemma}

\begin{proof}
Let $0<2Y<X$. Then
\begin{align*}
(X-Y)^{1/m} &\geq X^{1/m} - \frac{Y}{m}X^{1/m-1} - \frac{Y^2}{2m}\Bigl(1-\frac{1}{m}\Bigr)(X-Y)^{1/m-2} \\
 &\geq X^{1/m} - 2\frac{Y}{m}X^{1/m-1}
\end{align*}
for $m>1$, so we get
\begin{align*}
\#\{p\mid p^m\in [X,X-Y]\} &\leq X^{1/m}-(X-Y)^{1/m} + 1 \notag \\
&\leq  2 \frac{Y}{m}X^{1/m-1} + 1.
\end{align*}
For $X=e^\eps x$, $Y=2\sinh(\eps)x$ and $m\geq 2$ this is bounded by
\[
\frac{4.01}{m} \sqrt{x} + 1,
\]
and we get
\begin{equation*}
\sum_{\substack{p^m\in I\\m\geq 2}} \frac{1}{m} \leq  \int_{1}^{2\log(2x)} \frac{4.01 \eps \sqrt{x}}{t^2} + \frac{1}{t}\, dt
\leq 4.01 \eps \sqrt{x} + \log\log(2x^2).\qedhere
\end{equation*}
\end{proof}

\section{The explicit formula for $\pi^*_{c,\eps}(x)$}\label{s:exp-form}

\begin{theorem}\label{t:exp-form}
 Let $x>30000$, $c\geq 1$, and let $0<\eps\leq 0.01$. For $z\in \C\setminus [0,\infty)$ define the function
\begin{equation*}
 \Psi_{x,c,\eps}(z) = \lce \Bigl(\E_1(z\log x) + A_{c,\eps}\bigl(z \E_2(z\log x) - 2z^2\E_3(z\log x)\bigr)\Bigr)\ell_{c,\eps}\bigl(\tfrac{z}{i}- \tfrac{1}{2i}\bigr),
\end{equation*}
where 
\begin{equation}\label{e:Eik-def}
\E_k(\xi)  = \int_{0}^\infty \frac{e^{\xi-t}}{(\xi-t)^k}\, dt
\end{equation}
for $\xi\in \C\setminus [0,\infty)$.
Then we have
\begin{equation*}
 \pi_{c,\eps}^*(x) = \li(x) +\frac{ \Ace \, x}{\log(x)^2} -  \zsum_\rho \Psi_{x,c,\eps}(\rho) 
 - \log(2) + \int_x^\infty\frac{dt}{t\log t(t^2-1)} + \Theta(35\eps).
\end{equation*}
\end{theorem}

\begin{remark}
Since $\E_k(\rho\log x) \sim x^{\rho}/(\rho\log x)^k$ for $\abs{\Im(\rho)}\to \infty$, one
would expect 
\begin{equation*}
\sum_{\abs{\Im(\rho)}>c/\eps} \abs{\Psi_{x,c,\eps}(\rho)} \approx \frac{\sqrt{x}}{\pi \log x}\log\Bigl(\frac{c}{2\pi\eps}\Bigr)\int_{\abs{t}>c} \abs{\frac{\ell_c(t)}{t}}\, dt,
\end{equation*}
if the RH is assumed, where the optimality property \eqref{e:logan-opt} of the Logan function gets into play. Unconditionally, one would still expect the right hand side to give an upper bound 
if $\sqrt{x}$ is replaced by $x$. In Theorem \ref{t:zsum-bounds} we will prove bounds that are almost of this quality.
\end{remark}

\begin{proof}
For $0<\delta<1/x$ let 
\begin{equation*}
 f_{\delta,x} = \chi^*_{(\log\delta,\log x)} f_1, \eand g_{\delta,x} = \chi^*_{(\log\delta,\log x)}(f_2-2f_3),
\end{equation*}
and let  
\begin{equation*}
 F_{\delta,x}(t) = \half(f_{\delta,x}(t) + f_{\delta,x}(-t)), \eand G_{\delta,x}(t) = \half(g_{\delta,x}(t) + g_{\delta,x}(-t)).
\end{equation*}
We will prove the theorem by applying the Weil-Barner formula \eqref{e:WB-exp-form} to the function
\begin{equation*}
 H_{\delta,x,c,\eps} = \lce(F_{\delta,x} + A_{c,\eps} G_{\delta,x})\ast \nce
\end{equation*}
and taking the limit $\delta\searrow 0$.%

\begin{lemma}\label{l:ws}
 Let $x>30000$. Then we have
\begin{multline*}
 \lim_{\delta\searrow 0}\bigl( w_s(\F H_{\delta,x,c,\eps}) - \log\abs{\log\delta}\bigr)= -\li(x)  - \log\log(x) -\Ace\frac{x}{\log(x)^2} \\+\zsum_\rho \Psi_{x,c,\eps}(\rho) + \Theta\left(\frac{3.2\Ace}{\log x}\right).
\end{multline*}
\end{lemma}

\begin{proof}
We start by investigating the sum over zeros in $w_s(\F H_{\delta,x,c,\eps})$. Let $B_\delta\in\{F_{\delta,x},G_{\delta,x}\}$. Then partial integration shows $\F B_\delta(\xi) < C/\abs{\Re(\xi)}$ for $\abs{\Re(\xi)}>1$, $\abs{\Im(\xi)} \leq 1/2$ and $\delta$ sufficiently small. Furthermore, we have $\lim_{\delta\searrow 0} \hat B_\delta(\xi) = \hat B_0(\xi)$ for such $\xi$, where $B_0 = \lim_{\delta\searrow 0} B_\delta$. Therefore, we get
\begin{equation*}
\lim_{\delta\searrow 0} \zsum_\rho (\F B_\delta\cdot \ell_{c,\eps}) \Bigl(\frac{\rho-1/2}{i}\Bigr) = \zsum_\rho(\F B_0 \cdot \ell_{c,\eps}) \Bigl(\frac{\rho-1/2}{i}\Bigr).
\end{equation*}
\end{proof}

For $\Im(\xi)\neq 0$, the Fourier transforms of $F_{\delta,x}$ and $G_{\delta,x}$ may be expressed in terms of the functions $ \E_k$ defined in \eqref{e:Eik-def}. Using the abbreviation $z = 1/2+i\xi$ and $\tilde z = 1/2-i\xi$ we have
\begin{equation}\label{e:Fx-hat}
\F F_{0,x}(\xi) = \half \bigl(\E_1(z\log x) + \E_1(\tilde z\log x)\bigr),
\end{equation}
and
\begin{equation}\label{e:Gx-hat}
 \F G_{0,x}(\xi) = \frac{z}{2} \E_2(z\log x) - z^2\E_3(z\log x) + \frac{\tilde z}{2}\E_2(\tilde z\log x) - \tilde{z}^2\E_3(\tilde z\log x),
\end{equation}
which can be seen as follows. Define $g(t) = (f_2(t)-2f_3(t))(e^{i\xi t} + e^{-i\xi t})/2$, then we clearly have
\begin{equation*}
\F G_{0,x}(\xi) = \lim_{\eps\searrow 0} \Bigl(\int_{-\infty}^\eps + \int_{\eps}^{\log x}\Bigr) g(t) \, dt = \int_{-\infty}^{\log x} g(t+ir) + g(t-ir) \, dt + O(r).
\end{equation*}
On the other hand, the substitution $\tau = z(\log(x) - t)$ and a suitable modification of the resulting path of integration gives 
\begin{align*}
\int_{-\infty}^{\log(x)} \frac{e^{z(t\pm ir)}}{(t\pm ir)^k}\, dt &= z^{k-1} \int_{0}^\infty \frac{e^{z(\log(x)\pm ir) - \tau}}{(z(\log(x)\pm ir) - \tau)^k}\, d\tau \\
&= z^{k-1} \tilde \Ei_k( z(\log(x) \pm ir)).
\end{align*}
So the limit $r \searrow 0$ yields \eqref{e:Gx-hat}, and \eqref{e:Fx-hat} follows in a similar way. Now, since $\rho\mapsto 1-\rho$ is a bijection of the non-trivial zeros of the Riemann zeta function, \eqref{e:Fx-hat} and \eqref{e:Gx-hat} together imply
\[
\lim_{\delta\searrow 0} \zsum_\rho \F H_{\delta,x,c,\eps} \Bigl(\frac{\rho-1/2}{i}\Bigr) = \zsum_{\rho} \Psi_{x,c,\eps}(\rho).
\]

It remains to evaluate the term $\F H_{\delta,x,c,\eps}(i/2) + \F H_{\delta,x,c,\eps}(-i/2)$. We have
\begin{equation*}
 \Fdx(i/2) + \Fdx(-i/2) = \li(x) + \log\log(x) - \log\abs{\log \delta} + O(\delta) 
\end{equation*}
and
\begin{multline}\label{e:G-poles}
 \Gdx(i/2) + \Gdx(-i/2) = \int_{-\log x}^{\log x} \frac{e^t-1}{t^2} - 2\frac{e^t-1}{t^3}\,dt \\
 + \Theta\left(\int_{\log x}^\infty \Bigl(\frac{1}{t^2}+\frac{2}{t^3}\Bigr)(1+e^{-t})\right) + O\bigl(\abs{\log\delta}^{-1}\bigr),
\end{multline}
where the first integral on the right hand side of \eqref{e:G-poles} equals 
\[
 \left[\frac{e^t+t-1}{t^2}\right]_{-\log x}^{\log x} = \frac{x}{\log(x)^2} + \Theta\Bigl(\frac{2.1}{\log x}\Bigr),
\]
and where the integral in the $\Theta$-term is bounded by $1.1/\log(x)$. We therefore get
\begin{equation*}
\lim_{\delta\searrow 0} \Bigl(  \Gdx(i/2) + \Gdx(-i/2) \Bigr) = \frac{x}{\log(x)^2} + \Theta\Bigl(\frac{3.2}{\log x}\Bigr),
\end{equation*}
which yields the assertion.


\begin{lemma}\label{l:wf}
 Let $x>1$, $\eps\leq 0.01$ and $c\geq 1$. Then we have
\begin{equation*}
 \lim_{\delta\searrow 0}\bigl( w_f(H_{\delta,x,c,\eps}) - \log\abs{\log\delta}\bigr) = \gamma - \pi_{c,\eps}^*(x) + \Theta\Bigl(440 \frac {\eps^4}{ c^2}\Bigr).
\end{equation*}
\end{lemma}

\begin{proof}
We will prove the assertion by comparison of $w_f(H_{\delta,x,c,\eps})$ and $w_f(\Fdx)$, using the result
\begin{equation} \label{e:wf-F}
\lim_{\delta\searrow 0} \bigl(w_f(F_{\delta,x}) - \log\abs{\log\delta}\bigr) = \gamma - \pi^*(x)
\end{equation}
from \cite[Lemma 3.4]{BFJK13}.

For $\abs{t}>\eps$ let
\begin{equation*}
h(t) = \lce(\chi_{(\log\delta,\log x)}(f_1 + \Ace(f_2-2f_3)))\ast \nce(t).
\end{equation*}
Then we have
\[
H_{\delta,x,c,\eps}(t) = \half(h(t) + h(-t))
\]
for such $t$,  and since we assume $\eps\leq 0.01 <\log 2$ we get
\begin{equation}\label{e:wf-diff}
w_f(H_{\delta,x,c,\eps} - F_{\delta,x}) = \pi^*(x) - \pi_{c,\eps}^*(x) 
- \sum_{p^m} \frac{\log p}{p^{m/2}}\bigl(h - f_{\delta,x}\bigr)(-m\log p).
\end{equation}
The assertion thus follows from \eqref{e:wf-F}, if we show that the sum over prime powers in \eqref{e:wf-diff} is bounded by $440\eps^2/c^2$.

To this end let $g(t) = (h - f_{\delta,x})(-t)$ and $y=\delta^{-1}$. Then $g(t)$ vanishes for $t>y+\eps$. Since we have 
$
\abs{g(t)} \ll_\eps \frac{1}{\sqrt y\log y}
$
for $t\in B_\eps(y)$ the contribution of summands with $m\log p\in B_\eps(y)$ vanish for $\delta\to 0$ by the Brun-Titchmarsh theorem.
For the remaining summands we use the bound
\begin{equation*}
\abs{g(t)} = \abs{(\phi_{\infty,c,\eps}-f_1)(-t)} \leq 39\frac{\eps^4 }{c^2}\frac{e^{-t/2}}{t^2}
\end{equation*}
from Lemma \ref{l:phi-infty}, which gives
\begin{equation}\label{e:conv-psum}
\sum_{p^m\leq y-\eps} \frac{\log p}{p^{m/2}}\abs{g(m\log p)} \leq 39 \frac{\eps^4}{c^2} \sum_{p^m} \frac{1}{m^2 p^m \log p} \leq 39 \frac{\eps^4}{c^2}\zeta(2)  \sum_{p^m} \frac{1}{p \log p}.
\end{equation}
The Brun-Titchmarsh inequality gives
\begin{equation*}
\sum_{2^k\leq p < 2^{k+1}} \frac{1}{p\log p} < \frac{2^{k+1}}{k\log 2} \frac{2^{-k}}{k\log 2} < \frac{2}{\log(2)^2} k^{-2},
\end{equation*}
so the right hand side of \eqref{e:conv-psum} is indeed bounded by
\[
39  \,\zeta(2)^2 \frac{2}{\log(2)^2}\frac{\eps^4}{c^2} < 440\frac{\eps^4}{c^2}.\qedhere
\]
\end{proof}

\begin{lemma}\label{l:w8}
Let $x>30000$, $0<\eps<0.01$ and let $c\geq 1$. Then we have
\begin{equation*}
\lim_{\delta\searrow 0} w_\infty(H_{\delta,x,c,\eps}) = \int_{x}^\infty\frac{dt}{t\log t(t^2-1)} - \gamma - \log\log x - \log 2 + \Theta(34.9\,\eps).
\end{equation*}
\end{lemma}
\begin{proof}
We proceed again by comparing $w_\infty(H_{0,x,c,\eps})$ to $w_\infty(F_{0,x})$, using the result
\begin{equation}\label{e:w8-F}
\lim_{\delta\searrow 0} w_\infty(F_{\delta,x}) = \int_x^\infty \frac{dt}{t\log t(t^2-1)} - \gamma - \log 2
\end{equation}
from \cite[Lemma 3.5]{BFJK13}. So let $\Delta = H_{0,x,c,\eps}-F_{0,x}$. We aim to prove the estimates

\begin{align}
&-2\int_0^{\log 2} \frac{\Delta(t)-\Delta(0)}{1-e^{-2t}}e^{-t/2}\,dt = \Theta(14.1\,\eps),\label{e:int1}\\
&2\Delta(0)\int_{\log 2}^\infty \frac{e^{-t/2}}{1-e^{-2t}}\,dt = \Theta(14.5\,\eps),\label{e:int2}\\
&-\int_{\log 2}^{\log x} \frac{(\phi_{\infty,c,\eps}-f_1)(t)}{1-e^{-2t}}e^{-t/2}\,dt = \Theta(76\,\eps^4),\label{e:int3}\\
&-\int_{\log x -\eps}^{\log x+\eps} \frac{(\phi_{x,c,\eps} - \chi_{[1,\log x]} \phi_{\infty,c,\eps})(t)}{1-e^{-2t}}e^{-t/2}\,dt = \Theta(0.12\,\eps),\label{e:int4}\\
&-\int_{\log 2}^\infty \frac{(\phi_{x,c,\eps}-f_1)(-t)}{1-e^{-2t}}e^{-t/2}\,dt = \Theta(38\,\eps^4),\label{e:int5}\\
\intertext{and}
&\Delta(0)\Bigl(\Digamma(1/4) - \log \pi\Bigr) = \Theta(20.6\,\eps),\label{e:int6}.
\end{align}
Their left hand sides are easily seen to sum to $w_\infty(\Delta)$, and since we assume $\eps < 0.01$ and the left hand sides of \eqref{e:int2} and \eqref{e:int6} carry opposite sign, these estimates give the desired bound
\begin{equation*}
\abs{w_\infty(\Delta)} \leq \bigl(14.1 + 0 + 76\times 10^{-6} + 0.12 + 38\times 10^{-6} + 20.6\bigr)\, \eps < 34.9\,\eps.
\end{equation*}

We start by estimating $\Delta(t)$ for $\abs{t}\leq \log 2$. To this end, we will apply the estimate
\begin{multline}\label{e:conv-estimate}
\frac{1}{\lambda_{c,\eps}} \int_{-\eps}^\eps \nce(\tau) f(z-\tau) \, d\tau = \frac{1}{\lambda_{c,\eps}} \int_{-\eps}^\eps \nce(\tau)\bigl(f(z) + \Theta(\abs{\tau}\nrm{f'}_{\infty,U})\bigr)\, d\tau \\
= \lce f(z) + \Theta(\eps \nrm{f'}_{\infty,U}) = f(z) + \Theta(\eps(\nrm{f'}_{\infty,U} + 0.5 \nrm{f}_{\infty,U}))
\end{multline}
for all $z$ satisfying $B_\eps(z)\subset U$, which follows from
\begin{equation}\label{e:lce-bounds}
1 > \lce \geq \frac{\sinh(c+\eps/2)}{c+\eps/2} \geq e^{-\eps/2} \geq 1-\frac{\eps}{2},
\end{equation}
to the functions
\begin{equation*}
F(z) = \frac{\sinh(z/2)}{z} \eand G(z) = \frac{\cosh(z/2)}{z^2} - 2\frac{\sinh(z/2)}{z^3}.
\end{equation*}
These coincide with $F_{0,x}$ resp. $G_{0,x}$ at real numbers in $U_1=\{ z\in\C\mid \abs{\Ren(z)} < 3/2, \abs{\Imn(z)}< 3/2\}$. Since
\begin{equation*}
\max_{z\in \partial U_1} \lset \abs{\sinh(z/2)},\abs{\cosh(z/2)}\rset  \leq e^{3/4} < 2.2,
\end{equation*}
the maximum principle gives
\begin{equation*}
\abs{F(z)} \leq 2.2\,\frac{2}{3} < 1.5 \eand \abs{G(z)} \leq 2.2\Bigl(\frac{4}{9} + \frac{16}{27}\Bigr) < 2.3
\end{equation*}
for $z\in U_1$, s the Cauchy formula implies
\begin{equation*}
\Phi'(z) = \frac{1}{2\pi i}\int_{\abs{z-\xi}=\half}\frac{\Phi(\xi)}{(\xi-z)^2}\,d\xi =
\begin{cases}
\Theta(3) &\Phi=F,\\
\Theta(4.6) &\Phi = G
\end{cases}
\end{equation*}
for $z\in U_2=\{ z\in\C\mid \abs{\Ren(z)} < 1, \abs{\Imn(z)}< 1\} $.
We therefore get
\begin{equation*}
\lce F\ast\nce(z) = F(z) + \Theta(3.75\,\eps) \eand \lce G\ast \nce(z) = G(z) + \Theta(5.75\,\eps)
\end{equation*}
for $z\in U_3= \{z\in\C\mid \abs{\Ren(z)} < \log 2, \abs{\Imn(z)}<\pi/4\}$ from \eqref{e:conv-estimate}, which yields
\begin{equation}\label{e:Delta-bound}
\Delta(z) = \Theta(3.8\,\eps)\quad\quad\text{for $z\in U_3$.}
\end{equation}

Now in order to prove \eqref{e:int1} we observe 
\[
\abs{1-e^{-2z}} \geq 1-e^{-2\log 2} = \frac{3}{4}
\]
for $z\in \partial U_3$, so the maximum principle gives
\[
2 \int_0^{\log 2} \abs{\frac{\Delta(t)-\Delta(0)}{1-e^{-2t}}} e^{-t/2}\, dt \leq 2 \cdot \log 2 \cdot \Bigl(\frac{4}{3}\cdot 2\cdot 3.8\, \eps\Bigr) < 14.1 \, \eps.
\]

The estimate \eqref{e:int2} follows from
\[
\int_{\log 2}^\infty \frac{e^{-t/2}}{1-e^{-2t}} \, dt \leq \frac{4}{3} \int_{\log 2}^\infty e^{-t/2}\, dt = \frac{4}{3}\sqrt 2 < 1.9
\]
and \eqref{e:Delta-bound}.

For \eqref{e:int3} we use \eqref{e:phi-infty}, which gives
\[
\int_{\log 2}^{\log x} \frac{\abs{\phi_{\infty,c,\eps}(t) - f_1(t)}}{1-e^{-2t}}e^{-t/2}\, dt \leq 39\,\eps^4\frac{4}{3}\int_{\log 2}^\infty \frac{dt}{t^2} < \frac{39\cdot 4}{3\log 2} \,\eps^4 <  76\,\eps^4.
\]

For \eqref{e:int4} we use the bound
\begin{equation*}
\abs{m_{x,c,\eps}(t)} \leq \frac{e^{t/2}}{t}\Bigl(\half + \eps\Bigr).
\end{equation*}
which follows from
\begin{equation}\label{e:munu-bounds}
\abs{\mu_{c,\eps}(t)} \leq \half,\quad\quad \abs{\nu_{c,\eps}(t)} \leq \eps, \eand \abs{y}\leq \eps.
\end{equation}
Together with Lemma \ref{l:phi-at-logx} this gives
\[
\abs{(\phi_{x,c,\eps} - \chi^*_{[1,\log x]} \phi_{\infty,c,\eps})(t)} \frac{e^{-t/2}}{1-e^{-2t}} \leq 1.001\Bigl( \frac{\half+ \eps }{\log x - \eps} + 0.1 \frac{\eps^2}{\log x}\Bigr) < 0.06
\]
for $t\in B_\eps(\log x)$, which implies \eqref{e:int4}.

For \eqref{e:int5} we use \eqref{e:phi-infty} again, which gives
\begin{equation*}
\int_{\log 2}^\infty \frac{\abs{(\phi_{x,c,\eps}-f_1)(-t)}}{1-e^{-2t}}e^{-t/2}\,dt \leq 39\,\eps^4 \frac{4}{3} \frac{1}{2} \int_{\log 2}^\infty \frac{dt}{t^2} < 38\, \eps^4.
\end{equation*}

Finally, the estimate \eqref{e:int6} follows from \eqref{e:Delta-bound} and $\Gamma'/\Gamma (1/4) - \log \pi = \Theta(5.4)$.
\end{proof}

The assertion of Theorem \ref{t:exp-form} now follows from the previous lemmas, since we have $\log x > 10, $ $c\geq 1$, $\eps \leq 0.01$ and $\Ace \leq 0.005 \,\eps$, whence we find
\begin{equation*}
\Ace\frac{3.2}{\log x} + \frac{440\,\eps^4}{c^2} + 34.9 \, \eps   < 35\, \eps
\end{equation*}
for the sum of the $\Theta$-terms.
\end{proof}

\section{Estimates for the remainder terms}\label{s:zsum}

\subsection{Truncating the sum over zeros}
We provide two bounds for the tails of the sum over zeros:

\begin{theorem}\label{t:zsum-bounds}
Let $c\geq 10,$ $\eps \leq 10^{-5},$ and $x\geq e$. Let $h=\half$ if the RH is assumed and $h=1$ otherwise. Then we have
\begin{equation}\label{e:zsum-remainder}
\sum_{\substack{\rho\\\abs{\Im(\rho)}>c/\eps}} \abs{ \Psi_{x,c,\eps}(\rho)} \leq
0.66\,e^{c(\sqrt{\eps}/4-1)}\log(3c)\log\Bigl(\frac{c}{\eps}\Bigl) \frac{x^h+1}{2h \log x}.
\end{equation}
If in addition $a\in(0,1)$ satisfies $ac/\eps \geq 10^3$, and if the RH holds for $\abs{\Im(\rho)} \leq \frac{c}{\eps},$ then we have
\begin{equation}\label{e:zsum-part}
\sum_{\substack{\rho\\a\frac{c}{\eps}<\abs{\Im(\rho)}\leq \frac{c}{\eps}}} \abs{\Psi_{x,c,\eps}(\rho)} \leq
\frac{0.33 + 3.6\,c\,\eps}{c\, a^2}\log\Bigl(\frac{c}{\eps}\Bigl)\frac{\cosh(c\,\sqrt{1-a^2})}{\sinh(c)} \frac{\sqrt{x}}{\log x}.
\end{equation}
\end{theorem}

\begin{remark}\label{r:zsum-reduction}
If we assume $x^{-\beta}  < \eps < x^{-\alpha}$ for some $0<\alpha<\beta < 1$, we may choose $c=h\log(x) + \log\log\log(x) + C_{\alpha,\beta,\delta}$ (with $h$ as in the theorem) in order to make the right hand side of \eqref{e:zsum-remainder} smaller than any given $\delta$. Therefore, unconditional calculations require about twice as many zeros as calculations assuming the RH (if $\eps$ is left unchanged). If we choose $a=\sqrt{3/4}$ the right hand side of \eqref{e:zsum-part} is $o(1)$ for $x\to \infty$ for unconditional calculations, so the truncation bound in the sum over zeros can be reduced asymptotically by this factor if partial knowledge of the RH is available.
\end{remark}

The proof of the theorem needs some preparation. First we give asymptotic expansions for the functions $\E_k$ defined in \eqref{e:Eik-def}.

\begin{lemma}\label{l:ek-asymp}
 Let $\Im(z)\neq 0,$ $k\geq 1,$ and $n\geq -1$. Then
\begin{equation}\label{e:ek-asymp}
\E_k(z) = \sum_{l=0}^{n} \frac{(k+l-1)!}{(k-1)!} \frac{e^z}{z^{l+k}} + \Theta\Bigl( \frac{(k+n)!}{(k-1)!}\frac{e^{\Re(z)}}{\abs{\Im(z)}^{k+n+1}}\Bigr).
\end{equation}
\end{lemma}

\begin{proof}
By repeated integration by parts we find
\begin{equation*}
\E_k(z) = \int_{0}^\infty \frac{e^{z-t}}{(z-t)^k}\, dt = \sum_{j=0}^{n} \frac{(j+k-1)!}{(k-1)!} \frac{e^z}{z^{j+k}} + \frac{(k+n)!}{(k-1)!}\int_{0}^\infty \frac{e^{z-t}}{(z-t)^{k+n+1}}\, dt,
\end{equation*}
and the integral on the right hand side is bounded by
\begin{equation*}
\frac{1}{\abs{\Im(z)}^{k+n+1}}\int_0^\infty e^{\Re(z)-t}\, dt = \frac{e^{\Re(z)}}{\abs{\Im(z)}^{k+n+1}}.
\end{equation*}
\end{proof}

Next, we need some bounds for sums over zeros involving the Logan function. From \cite{FKBJ} we quote
\begin{lemma}[{\cite[Lemma 2.4]{FKBJ}}]\label{l:zsum1}
 Let $0<\eps<10^{-5}$ and let $c\geq 10$. Then we have
\begin{equation*}
 \sum_{\substack{\rho\\\abs{\Im(\rho)} > \frac{c}{\eps}}} \frac{\abs{\ell_{c,\eps}(\frac{\rho}{i}-\frac1{2i})}}{\abs{\Im(\rho)}} \leq 0.65 e^{c(\sqrt\eps /4 - 1)}\log(3c) \log(c/\eps).
\end{equation*}
\end{lemma}

Furthermore, we need
\begin{lemma}\label{l:zsum-logan-part}
Let $c,\eps>0$, and $a \in(0,1)$ satisfy $\frac{ac}{\eps}\geq 10^3$. Then we have
\begin{equation*}
\sum_{\substack{\rho \\ \frac{ac}{\eps}< \abs{\Im(\rho)}\leq \frac{c}{\eps}}} \abs{\frac{\ell_{c,\eps}(\Im(\rho))}{\Im(\rho)}} 
\leq \frac{1+11\,c\,\eps}{\pi c\, a^2} \log\Bigl(\frac{c}{\eps}\Bigr)\frac{\cosh(c\,\sqrt{1-a^2})}{\sinh(c)}.
\end{equation*}
\end{lemma}
\begin{proof}
We denote zeros of the Riemann zeta function by $\rho = \beta + i\gamma$ with $\beta,\gamma \in \R$. Let
$N(t)$ denote the number of zeros of the Riemann zeta function (counted according to their multiplicity) with imaginary part in $(0,t]$ and let
\[
\tilde N(t) = \frac{t}{2\pi}\log \frac{t}{2\pi e} \eand R(t) = N(t) - \tilde N(t).
\]
Then Rosser's estimate \cite[p. 223]{Rosser41} implies 
\begin{equation}\label{e:R-bound}
R(t)= \Theta(0.5\log t)
\end{equation}
 for $t\geq 10^3$. By symmetry of the zeros, it suffices to treat the sum over $\gamma>0$.

We have
\begin{align}
\sum_{a\frac{c}{\eps} < \gamma \leq \frac{c}{\eps}} \frac{\ell_{c,\eps}(\gamma)}{\gamma} 
&= \int_{a\frac{c}{\eps}}^\frac{c}{\eps} \frac{\ell_{c,\eps}(t)}{t}\, d(\tilde N(t) + R(t)) \notag \\
&= \frac{1}{2\pi} \int_{a\frac{c}{\eps}}^\frac{c}{\eps} \ell_{c,\eps}(t) \log \frac{t}{2\pi} \frac{dt}{t} + \int_{a\frac{c}{\eps}}^\frac{c}{\eps}\frac{\ell_{c,\eps}(t)}{t}\, d R(t).\label{e:zsum-part-1}
\end{align}
First, we estimate the first integral on the right hand side of \eqref{e:zsum-part-1}. Using the inequality
\[
0 < \frac{1}{t}\log \frac{t}{2\pi} \leq \frac{\eps^2}{(ac)^2}\log\Bigl(\frac{c}{2\pi \eps}\Bigr) t
\]
and applying the substitution $u= \sqrt{c^2-(\eps t)^2}$ gives the bound
\begin{align}
0 < \frac{1}{2\pi} \int_{a\frac{c}{\eps}}^\frac{c}{\eps} \ell_{c,\eps}(t) \log \frac{t}{2\pi} \frac{dt}{t}
&\leq\frac{1}{2\pi (ac)^2} \log\Bigl(\frac{c}{2\pi \eps}\Bigr) \frac{c}{\sinh(c)}\int_{0}^{c\sqrt{1-a^2}}\!\!\!\!\!\! \sinh(u)\, du \notag \\
&\leq \frac{1}{2\pi ca^2} \log\Bigl(\frac{c}{2\pi \eps}\Bigr)\frac{\cosh(c\,\sqrt{1-a^2})}{\sinh(c)}.\label{e:zsum-part-2}
\end{align}

For the second integral in \eqref{e:zsum-part-1} we use partial integration, the bound from \eqref{e:R-bound} and the negativity of the derivative of $\ell_{c,\eps}(t)/t$ in the range of integration, which gives
\begin{align}
\int_{a\frac{c}{\eps}}^\frac{c}{\eps} \frac{\ell_{c,\eps}(t)}{t}\,dR(t)
&= \Theta\left(\left[\frac{\ell_{c,\eps}(t)}{t}R(t)\right]_{a\frac{c}{\eps}}^\frac{c}{\eps} - \int_{a\frac{c}{\eps}}^\frac{c}{\eps}\frac{d}{dt}\Bigl(  \frac{\ell_{c,\eps}(t)}{t}\Bigr) R(t)\, dt\right)\notag \\
&\leq  \frac{\eps}{ac} \ell_c(a c)\log\Bigl(\frac{c}{\eps}\Bigr) - 0.5 \int_{a\frac{c}{\eps}}^\frac{c}{\eps} \frac{d}{dt}\Bigl(  \frac{\ell_{c,\eps}(t)}{t}\Bigr) \log(t)\, dt\notag \\
&\leq 1.5 \frac{\eps}{ac} \ell_c(ac)\log\Bigl(\frac{c}{\eps}\Bigr) + 0.5 \int_{a\frac{c}{\eps}}^\frac{c}{\eps} \frac{\ell_{c,\eps}(t)}{t^2}\, dt \, \leq \, 1.6\frac{\eps}{ac} \ell_c(ac)\log\Bigl(\frac{c}{\eps}\Bigr),
\end{align}
where we also used $\log(c/\eps)>6$ on the last line. Since the function $t\mapsto \frac{t\cosh t}{\sinh t}$ is monotonically increasing in $[0,\infty),$ we have
\[
c\sqrt{1-a^2} \;\frac{\cosh(c\sqrt{1-a^2}  )}{\sinh( c\sqrt{1-a^2} )}\geq 1
\]
and therefore get
\[
1.6\frac{\eps}{ac} \log\Bigl(\frac{c}{\eps}\Bigr) \ell_c(ac) \leq 1.6\frac{\eps c}{ac} \log\Bigl(\frac{c}{\eps}\Bigr) \frac{\cosh(c\sqrt{1-a^2})}{\sinh (c)}.
\]
This and \eqref{e:zsum-part-2} yields the assertion, since we have
\[
  \frac{1}{2\pi ca^2} \log\Bigl(\frac{c}{2\pi \eps}\Bigr) + 1.6\frac{\eps c}{ac} \log\Bigl(\frac{c}{\eps}\Bigr)
\leq \frac{1}{2} \frac{1+11\,c\,\eps}{\pi ca^2}\log\Bigl(\frac{c}{\eps}\Bigr).
\]
\end{proof}

\begin{proof}[Proof of Theorem \ref{t:zsum-bounds}]
We recall that $\abs{\gamma}>14$ for all non-trivial zeros of the zeta function by well-known numerical results (see e.g. \cite{Lehman70}), so in particular we have $\abs{\rho/\gamma} \leq 1.08$. Consequently, Lemma \ref{l:ek-asymp} and \eqref{e:logan-der2} give the bound
\begin{align}
\abs{\Psi_{x,c,\eps}(\rho)} 
	&\leq \frac{x^\beta}{\abs{\gamma}\log x}\Bigl( 1 + \frac{\eps^2}{2c}\Bigl(\frac{1.08}{\log x} + \frac{2.4}{\log(x)^2}\Bigr)\Bigr)\abs{\ell_{c,\eps}\Bigl(\frac{\rho-1/2}{i}\Bigr)}\notag \\
	&\leq 1.001\frac{x^\beta}{\log(x)} \frac{\abs{\ell_{c,\eps}(\frac \rho i - \frac 1{2i})}}{\abs{\gamma}} .\label{e:Psi-bound}
\end{align}
Now the case $h=1/2$ in \eqref{e:zsum-remainder} follows directly from \eqref{e:Psi-bound}, where we may take $\beta=1/2$, and Lemma \ref{l:zsum1}. For the unconditional case $h=1$ we use
\begin{align*} 
\sum_{\abs{\Im(\rho)}>c/\eps} \frac{\abs{\ell_{c,\eps}(\frac \rho i - \frac 1{2i})}}{\abs{\gamma}}x^\beta 
	&= \sum_{\substack{\abs{\Im(\rho)}>c/\eps \\ \Re(\rho)=1/2}} \frac{\abs{\ell_{c,\eps}(\frac \rho i - \frac 1{2i})}}{\abs{\gamma}}\sqrt{x} \\
	 &\quad\quad\quad+ \sum_{\substack{\abs{\Im(\rho)}>c/\eps \\ \Re(\rho)>1/2}} \frac{\abs{\ell_{c,\eps}(\frac \rho i - \frac 1{2i})}}{\abs{\gamma}}(x^\beta + x^{1-\beta}) \\
	 &\leq \frac{x+1}{2}\sum_{\abs{\Im(\rho)}>c/\eps} \frac{\abs{\ell_{c,\eps}(\frac \rho i - \frac 1{2i})}}{\abs{\gamma}}.
\end{align*}
The bound \eqref{e:zsum-part} follows in a similar way from \eqref{e:Psi-bound} and Lemma \ref{l:zsum-logan-part}.
\end{proof}

\subsection{Further truncation of the sum over prime powers}
We provide an estimate for the remainder if the interval $[e^{-\eps}x,e^\eps x]$ in Theorem \ref{t:pice-to-pi} is further truncated to $[e^{-\alpha \eps}x, e^{\alpha\eps}x]$. This is particularly interesting for unconditional calculations. We abbreviate $\mu_{c,1}$ and $\nu_{c,1}$ by $\mu_c$ and $\nu_c$ respectively.

\begin{prop}\label{p:sieve-reduction}
Let $x\geq 100$, $\eps \leq 0.01$, $c\geq 1$, and let $\alpha\in(0,1)$, such that
\[
B := \frac{\eps \, x\, e^{-\eps}\abs{\nu_c(\alpha)}}{2\mu_c (\alpha)} > 1
\]
holds. Furthermore, let $I_\alpha^+ = [e^{\alpha\eps}x, e^\eps x]$ and $I_\alpha^- = [e^{-\eps}x, e^{-\alpha\eps}x]$. Then we have
\begin{equation}\label{e:sieve-reduction}
\abs{\sum_{p^m\in I_\alpha^\pm} \frac 1m M_{x,c,\eps}(p^m) } \leq \frac{2\eps \, x \, e^{2\eps} \abs{\nu_c(\alpha)}}{\log B} + e^\eps \mu_c(\alpha)\left(4.01\eps\sqrt{x} + \log\log(2 x^2)\right).
\end{equation}
\end{prop}

\begin{remark}
Since $I_0(t)\sim e^t/\sqrt{2\pi t}$, we have
\[e^{c\sqrt{1-\alpha^2} - c}/c^{1/2} \ll \mu_c(\alpha) \ll e^{c\sqrt{1-\alpha^2} - c}/c^{1/2}\]
and
$\abs{\nu_c(\alpha)} \ll e^{c\sqrt{1-\alpha^2} - c}/c^{3/2}$ uniformly for $\alpha\in(\delta,1-\delta)$ with any $\delta>0$. So for calculations not assuming the RH, where one would choose $c\geq \log x$ and $\eps=C\sqrt{\log(x)/x}$ the right hand side of \eqref{e:sieve-reduction} is $\ll x^{\sqrt{1-\alpha^2}}/(\sqrt{x}\log(x))$ for $x\to\infty$. So the interval $[e^{-\eps}x,e^\eps x]$ may be asymptotically reduced in length by a factor $\sqrt{3/4}$ for unconditional computations.
\end{remark}

We use the following lemma, which is based on a sieve bound for weighted sums over prime numbers from \cite{buethe14}.

\begin{lemma}\label{l:psum-mu}
Let $x>1,\eps<1$ and $\alpha\in(0,1)$, such that
\[
B := \frac{\eps x e^{-\eps}\abs{\nu_c(\alpha)}}{2\mu_c(\alpha)} > 1
\]
holds. Then we have
\begin{equation}\label{e:Ia-bound}
\sum_{ p \in I_\alpha^\pm} \abs{\mu_{c,\eps}\left(\log \frac p x\right)} \leq 2 \frac{\eps \, x \,  e^\eps \abs{\nu_c(\alpha)}}{\log B}.
\end{equation}
\end{lemma}

\begin{proof}
We give the proof for $I=I_\alpha^+$. For $t\in I$ let
\[
f(t) = \mu_c \Bigl(\frac 1\eps \log \frac t x\Bigr).
\]
Then $f$ satisfies the conditions of \cite[Theorem 4.3]{buethe14} and we therefore have
\[
\sum_{p\in I} f(p) \leq 2 \nrm{f}_{1,I}\Bigl(\log \frac{\nrm{f}_{1,I}}{\nrm{f}_{\infty,I} + \nrm{f'}_{1,I}}\Bigr)^{-1}.
\]
Since $f$ is monotonously decreasing on $I$, we have $\nrm{f}_{\infty,I} = \nrm{f'}_{1,I} = \mu_{c}(\alpha)$, and the substitution $u=\frac1\eps \log \frac t x$ yields
\[
\nrm{f}_{1,I} = \int_{e^{\alpha\eps}x}^{e^\eps x} f(t)\, dt 
= \eps x \int_{\alpha}^1 e^{\eps u} \mu_{c}(u)\, du \leq -\eps \, x \,  e^\eps \nu_c(\alpha),
\]
which yields \eqref{e:Ia-bound}. The interval $I_\alpha^-$ can be treated similarly.
\end{proof}

\begin{proof}[Proof of Proposition \ref{p:sieve-reduction}]
From \eqref{e:munu-bounds} we obtain the bound
\begin{equation*}
\abs{m_{x,c,\eps}(t) } \leq \frac {e^{t/2}} t\abs{\mu_{c,\eps}\bigl(t-\log(x)\bigr)}(1 + \eps)
\end{equation*}
for $t\neq \log(x)$. Therefore, we have
\begin{equation*}
\abs{\sum_{p^m\in I_\alpha^\pm} \frac 1m M_{x,c,\eps}(p^m) } \leq e^\eps \sum_{p^m\in I_\alpha^\pm}  \frac 1m \abs{\mu_{c,\eps}\left(\log\frac {p^m} x\right)}.
\end{equation*}
The assertion of Proposition \ref{p:sieve-reduction} now follows directly from Lemma \ref{l:psum-mu} and Lemma \ref{l:ppow-bound}.
\end{proof}

\section{Implementation}
The implementation of the method is very similar to the implementation of Method I in \cite{FKBJ}. Therefore, the description in this paper is kept short and focuses on different aspects.

\subsection{Evaluation of the sum over zeros}
We need to evaluate the function
\begin{equation*}
\Psi_{x,c,\eps}(\rho) = \psi_{x,c,\eps}(\rho) \ell_{c,\eps}\Bigl(\frac{\rho}{i} - \frac{1}{2i}\Bigr),
\end{equation*}
where
\begin{equation}\label{e:psi-def}
\psi_{x,c,\eps}(\rho) = \E_1(\rho\log x) + \Ace(\rho\E_2(\rho\log x) - 2\rho^2\E_3(\rho\log x))
\end{equation}
within an accuracy of $O(x^{-A})$ for some $A>1$. As a corollary of Lemma \ref{l:ek-asymp} we get the following asymptotic expansion for $\psi_{x,c,\eps}(\rho)$:
\begin{cor}\label{c:psi-comp}
Let $c\geq 10$ and $\eps\leq 0.01$. For $j\geq 1$, let
\begin{equation*}
\alpha_j(x,c,\eps) = \frac{(j-1)!}{\log(x)^j} + \Ace\Bigl(\frac{j!}{\log(x)^{j+1}} - \frac{(j+1)!}{\log(x)^{j+2}}\Bigr).
\end{equation*}
Furthermore, let $\abs{\Im(\rho)}>14$, $\Re(\rho)\in(0,1)$ and $n+2\leq 10\log x$. Then we have
\begin{equation}\label{e:psi-asymp}
\psi_{x,c,\eps}(\rho) =  \sum_{j=1}^{n} \alpha_j(x,c,\eps) \frac{x^\rho}{\rho^j} + \Theta\Bigl(1.01 n!\frac{x^{\Re(\rho)}}{(\abs{\Im(\rho)}\log x)^{n+1}}  \Bigr).
\end{equation}
\end{cor}
Choosing e.g. $n=\lfloor \log(x)\rfloor$ gives an error term of size $O(x^{-3})$ which is sufficiently small for practical applications.

\begin{proof}
The main term in Lemma \ref{l:ek-asymp} gives the main term in \eqref{e:psi-asymp} and the sum of the $\Theta$-terms is bounded by
\begin{multline}\label{e:psi-comp-1}
n!\frac{x^{\Re(\rho)}}{\abs{\Im(\rho)\log x}^n}\left[1 + A_{c,\eps}\left(\frac{(n+1)\abs{\rho}}{\abs{\Im(\rho)\log x}} + \frac{(n+1)(n+2)\abs{\rho}^2}{\abs{\Im(\rho)\log x}^2}\right)\right] \\
\leq 1.01 \frac{x^{\Re(\rho)}}{(\abs{\Im(\rho)}\log x)^{n+1}},
\end{multline}
where we used
\begin{equation*}
\frac{\abs{\rho}}{\abs{\Im(\rho)}} \leq \frac{1+14}{14} < 1.08,\quad\quad \frac{n+2}{\log(x)} \leq 10,\quad\quad\text{and}\quad A_{c,\eps} \leq 5\times 10^{-6}.
\end{equation*}
\end{proof}

For the purpose of calculating $\Psi_{x,c,\eps}(\rho)$ one may safely assume $\Re(\rho)=1/2$. The evaluation of $\Psi_{x,c,\eps}(\rho)$ can then be sped up by using piecewise Chebyshov approximation on the slowly varying function
\[
\Psi_{x,c,\eps}(1/2+it) x^{-it}.
\]
This way the most time consuming part in calculating $\Psi(1/2+i\gamma)$ is the evaluation of $\exp(i\gamma\log x)$.

\subsection{Evaluation of the sum over prime powers}
We intend to calculate the sum
\begin{equation}\label{e:psum}
\sum_{p^m\in I} \frac{1}{m}M_{x,c,\eps}(p^m),
\end{equation}
where $I = [e^{-\alpha\eps}x, e^{\alpha\eps}x]$, within an accuracy $<\delta$. To this end it suffices to evaluate
\begin{equation}\label{e:Mxce-implement}
M_{x,c,\eps}(t) = \lce\Bigl[\mu_{c,\eps}\bigl(\log\tfrac{t}{x}\bigr) + \Bigl(\frac{1}{\log t} - \frac{1}{2}\Bigr)\Bigl(\mu_{c,\eps}\bigl(\log\tfrac{t}{x}\bigr)\log\tfrac{t}{x} - \nu_{c,\eps}\bigl(\log\tfrac{t}{x}\bigr)\Bigr)\Bigr]
\end{equation}
within an accuracy of $O(x^{-A})$ for some $A>1$. This can be done using the power series expansion
\[
\eta_c(t) = \chi^*_{[-1,1]} \sum_{k=0}^\infty \lambda_k t^{2k}
\]
from \cite[Section 4.2]{FKBJ}, where
\[
\lambda_k = -\frac{(c/2)^{k+1}}{k!\sinh(c)}I_k(c),
\]
and where
\[
I_k(c) = \sum_{n=0}^\infty \frac{(c/2)^{2n+k}}{n! (n+k)!}
\]
denotes the $k$-th modified Bessel function of the first kind. Since $\lim_{h\searrow 0}\mu_c(h) = 1/2$ and $\nu_c(0) = -\frac{I_1(c)}{2\sinh(c)}$, this gives
\begin{equation}\label{e:mu-expansion}
\mu_c(t) = \frac 12 - \sum_{k=0}^\infty \frac{\lambda_k}{2k+1} t^{2k+1} 
\end{equation}
and
\begin{equation}\label{e:nu-expansion}
\nu_{c}(t) = -\frac{I_1(c)}{2\sinh(c)} + \frac{t}{2} - \sum_{k=0}^\infty \frac{\lambda_k}{(2k+1)(2k+2)} t^{2k+2} 
\end{equation}
for $t\in(0,1)$. On the reasonable assumption that $\log(x) < 2c \leq 4\log(x)$ it follows from the considerations in \cite[Section 4.2]{FKBJ} that the error from truncating these series at $k = \lceil ce\rceil$ is $< 6\log (x) x^{-1.8}$, which suffices for all practical applications.

For the calculation of $p^m\in I$ with $m\geq 2$ the prime powers $p^m$ are enumerated with the Eratosthenes sieve and $M_{x,c,\eps}(p^m)$ is evaluated directly, using the power series \eqref{e:mu-expansion} and \eqref{e:nu-expansion}. For the calculation of the contribution of prime numbers in $I$ the interpolation techniques from \cite{FKBJ} are used. The interval $I$ is dissected into subintervals $I_k = [a_k,b_k]$ of length $\leq 2^{20}$ and the sum over $p\in I_k$ is approximated from the data
\begin{equation}\label{e:sieve-data}
s_{k,l} = \sum_{p\in I_k} (p-a_k)^l
\end{equation}
using quadratic interpolation. For calculations with $x>10^{20}$ the space requirement of the Eratosthenes sieve is currently reduced by calculating the sum as
\begin{equation*}
\sum_{p\in I} M_{x,c,\eps}(p) = \sum_{\substack{n\in I\\p\mid n \Rightarrow p> B}} M_{x,c,\eps}(n) - \sum_{\substack{p q\in I\\  p,q \geq B}} M_{x,c,\eps}(pq),
\end{equation*}
where $(e^{\alpha\eps}x)^{1/3} < B < e^{-\alpha\eps}x$, which is outlined in detail in \cite{FKBJ}. This way the the space requirement for the sieving process can be reduced to $O(x^{1/3+\delta})$ but at the expense of  increasing the run time to $O(x^{2/3+\delta})$ (assuming $\eps \ll x^{\delta' - 1/2}$). This can be avoided by implementing the dissected Atkin-Bernstein sieve described in \cite{Galway04}, which has been carried out in \cite{Platt15}.

\subsection{Numerical calculations}
For the unconditional calculations of $\pi(10^{24})$ and $\pi(10^{25})$ we both took $c=62$ and
$\eps = 6.2\times 10^{-10}$ and computed the sum for $\abs{\Im(\rho)}$ up to $10^{11}$. The calculation of
 $\pi(10^{24})$ took less than $3,900$ hours and the calculation of $\pi(10^{25})$ took less than $40,000$
 on $2.27$ GHz Intel Xeon X7560 CPUs provided by the Hausdorff Center for Mathematics in Bonn. Evaluating the sum over zeros took less than $100$ hours so most of the time was spent sieving the intervals of lengths $1.24\times 10^{15}$ resp. $1.24\times 10^{16}$ around $10^{24}$ and $10^{25}$. With the same amount of zeros both run times could have been reduced by approx. $32\%$ choosing $a = 0.8$ in Theorem \ref{t:zsum-bounds} and $\alpha = 0.84$ in Proposition \ref{p:sieve-reduction}, and reducing $\eps$ to $4.96\times 10^{-10}$. This has not been done since these results were not available when the calculations were started. An even further reduction of the run times could have been achieved by using additional zeros of the zeta function. For these calculations it would have been optimal to use all zeros with imaginary part up to $6\times 10^{11}$ resp. $2\times 10^{12}$, which by linear extrapolation projects to run times of $1,300$ resp. $4,000$ hours.

\section{Conclusion}\label{s:calculation}
The present method and the analytic methods in \cite{FKBJ} and \cite{Platt15} can be compared quite well, since the arithmetic mean of the truncation bound in the sum over zeros and the length of the interval around $x$ in the sum over prime powers is
\begin{equation}\label{e:gm-asymp}
\sim C \sqrt{x\log x}
\end{equation}
for $x\to\infty$, which has already been mentioned in \cite{FKBJ}. The constant $C$ gives a good measure for the method's efficiency, since the run time for calculating the sums in question for different methods is practically determined by the number of summands if interpolation techniques are used. With the improvements from this paper we get the admissible values stated in Table \ref{tbl:C-values}, where the new method is referred to as Method III, and where the assumption \textit{partial RH} means, that the Riemann Hypothesis is known up to a small multiple of the truncation bound.

\renewcommand{\arraystretch}{1.2}
\vspace*{.2cm}
\begin{center}
\begin{table}[h]
\caption{Admissible values for $C$.}\label{tbl:C-values}
\begin{tabular}{l|c|c|c}
method $\backslash$ assumptions				& unconditional	& partial RH	& RH\\
\hline
Galway-Platt	 \cite{Platt15,Galway04} 	&	$2^{3/4}$		& $2^{1/2}$	& $2^{1/2}$ \\
Method I/II	\cite{FKBJ}	& $2^{1/2}$		& $2^{1/2}	$  & $1$	\\
Method III & $3^{1/4}$ & $(3/2)^{1/2}$ &$1$ \\
\end{tabular}
\end{table}
\end{center}

\bibliographystyle{amsplain}

\bibliography{/home/jbuethe/TeX/inputs/jankabib.bib}
\end{document}